\documentclass[11pt]{amsart}

\usepackage[T1]{fontenc}
\usepackage[utf8]{inputenc}
\usepackage{geometry}
\usepackage{amssymb}
\usepackage{amsmath}
\usepackage{mathtools}
\usepackage{xcolor}
\usepackage[colorlinks=true,citecolor=red,linkcolor=blue,urlcolor=red]{hyperref}
\usepackage{soul}

\DeclareMathOperator{\Pic}{Pic}
\DeclareMathOperator{\Nef}{Nef}
\DeclareMathOperator{\Eff}{\overline{Eff}}
\DeclareMathOperator{\Sym}{Sym}
\DeclareMathOperator{\Proj}{Proj}
\DeclareMathOperator{\rank}{rank}

\newcommand{\PP}{\mathbb{P}}
\newcommand{\RR}{\mathbb{R}}
\newcommand{\ZZ}{\mathbb{Z}}
\newcommand{\Oo}{\mathcal{O}}
\newcommand{\Ee}{\mathcal{E}}

\newtheorem{thm}{Theorem}[section]
\newtheorem{lem}[thm]{Lemma}
\newtheorem{prop}[thm]{Proposition}

\theoremstyle{definition}

\newtheorem{ques}[thm]{Question}
\newtheorem{rem}[thm]{Remark}

\numberwithin{equation}{section}

\begin{document}

\title[Slope instability and \texorpdfstring{$1$}{1}-homogeneity on a surface]{A slope-unstable bundle on a surface with $1$-homogeneous projectivization}

\author{Jihao Liu}
\address{Department of Mathematics, Peking University, No. 5 Yiheyuan Road, Haidian District, Beijing 100871, China}
\address{Beijing International Center for Mathematical Research, Peking University, No. 5 Yiheyuan Road, Haidian District, Beijing 100871, China}
\email{liujihao@math.pku.edu.cn}

\date{}

\begin{abstract}
We construct a slope-unstable vector bundle on $\PP^2$ whose projectivization is $1$-homogeneous, as asked in a question of Fulger and Langer on the interplay between slope semistability and positivity for projective bundles with vanishing discriminant.  The main result of this paper was obtained by Chatgpt 5.5 pro, and the Danus system based on the Rethlas system.
\end{abstract}

\maketitle

\section{Introduction}\label{sec:introduction}

We work over the field $\mathbb C$ of complex numbers. For a projective variety $Y$, let $N^1(Y)_{\RR}$ denote the real N\'eron--Severi space of numerical divisor classes, let $\Nef(Y)\subseteq N^1(Y)_{\RR}$ be the cone of nef classes, and let $\Eff(Y)\subseteq N^1(Y)_{\RR}$ be the pseudo-effective cone, that is, the closure of the cone generated by the classes of effective divisors. One always has $\Nef(Y)\subseteq\Eff(Y)$.

We use the terminology of Fulger and Langer \cite[Definition~4.1]{FL22}. A projective variety $Y$ is called \emph{$1$-homogeneous} if
\begin{equation}\label{eq:1homog-def}
\Nef(Y)=\Eff(Y).
\end{equation}
Recall that a vector bundle $\Ee$ on a polarized variety $(X,H)$ is \emph{slope-unstable} (with respect to $H$) if it admits a nonzero proper subbundle $\mathcal F$ with $\mu_H(\mathcal F)>\mu_H(\Ee)$, where $\mu_H(\mathcal G)=(c_1(\mathcal G)\cdot H^{\dim X-1})/\rank(\mathcal G)$ denotes the $H$-slope. 

In their study of the relationship between positivity and slope semistability for vector bundles with vanishing discriminant, Fulger and Langer established the following result in every characteristic. We note that the characteristic $0$ case is established by \cite[Theorem~1.2]{Mis21}. 

\begin{thm}[{\cite[Theorem~4.3]{FL22}}]
Let $X$ be a smooth projective variety defined over an algebraically closed field $k$ and
let $\mathcal{E}$ be a strongly slope semistable bundle with respect to some ample polarization of $X$. Let us also
assume that the discriminant satisfies $\Delta(\mathcal{E}) \equiv 0$. Then the following conditions are equivalent:

(1) $X$ is $1$-homogeneous,

(2) $c_1(\pi_*\mathcal{O}_{\mathbb{P}(\mathcal{E})}(D))$ is nef for every effective divisor $D$ on $P(\mathcal{E})$,

(3) $P(\mathcal{E})$ is $1$-homogeneous.
\end{thm}

It is natural to ask whether this kind of result holds without the slope stability condition. 
To answer this question, Fulger and Langer constructed a three-dimensional base, namely $\PP^3$, carrying a slope-unstable bundle whose projectivization is $1$-homogeneous in \cite[Example~4.7]{FL22}, and asked whether the same phenomenon can occur over a surface. 

The following question is the second part of \cite[Question~2]{FL22}.

\begin{ques}[{\cite[Question~2(2)]{FL22}}]\label{ques:fl}
Does there exist a complex projective surface $X$ supporting a slope-unstable vector bundle $\Ee$ such that $\PP(\Ee)$ is $1$-homogeneous?
\end{ques}

The purpose of this note is to present the explicit example in Theorem~\ref{thm:main}.

\begin{thm}\label{thm:main}
Let $X=\PP^2_{\mathbb C}$, let $H=c_1(\Oo_{\PP^2}(1))$ be the hyperplane class, and let
\[
\Ee=T_{\PP^2}(-1)\oplus\Oo_{\PP^2}.
\]
Then $\Ee$ is a rank-three vector bundle on $X$ that is slope-unstable with respect to $H$, while its projectivization $\PP(\Ee)$ is $1$-homogeneous. In particular, Question~\ref{ques:fl} has an affirmative answer.
\end{thm}

\begin{rem}\label{rem:sarkar}
After the first version of this note appeared, Supravat Sarkar kindly pointed out that the argument of \cite[Theorems~7.3 and~7.4(i)]{BSV24} gives further examples. Let $\mathcal M=M_{\Oo_{\PP^2}(2)}$ denote the syzygy bundle defined by
\[
0\to \mathcal M
\to H^0(\PP^2,\Oo_{\PP^2}(2))\otimes\Oo_{\PP^2}
\to \Oo_{\PP^2}(2)\to 0.
\]
Then $\Ee'=\Oo_{\PP^2}\oplus\mathcal M^\vee$ is slope-unstable with respect to $H$, whereas $\PP(\Ee')$ is $1$-homogeneous. Indeed, $\rank(\mathcal M^\vee)=5$ and $c_1(\mathcal M^\vee)=2H$, so the summand $\mathcal M^\vee$ has slope $2/5>1/3=\mu_H(\Ee')$; the $1$-homogeneity follows from the cited results and the two-ray argument used below.
\end{rem}

Throughout, we use the Grothendieck convention $\PP_X(\Ee)=\Proj_X(\Sym\Ee)$, so that $\PP_X(\Ee)$ parametrizes one-dimensional quotient spaces of the fibers of $\Ee$ and carries the tautological quotient line bundle $\Oo_{\PP_X(\Ee)}(1)$.

\subsection*{Outline of the argument}
The two halves of Theorem~\ref{thm:main} are independent and elementary. The slope computation rests on the twisted Euler sequence on $\PP^2$, which identifies $T_{\PP^2}(-1)$ as a rank-two quotient of $\Oo_{\PP^2}^{\oplus 3}$ with first Chern class $H$; the destabilizing subbundle is the summand $T_{\PP^2}(-1)$ itself. The $1$-homogeneity is forced by a two-ray cone argument: the fourfold $Y=\PP(\Ee)$ has Picard number two, and the two natural divisor classes $A=\pi^*H$ and $B=c_1(\Oo_Y(1))$ are linearly independent, nef, and non-big, so they span both the nef cone and the pseudo-effective cone. This two-ray cone argument is mentioned in \cite[Remark 4.2]{FL22}.

\begin{rem}
The sketch of the proof of the main result of this paper was obtained by Chatgpt 5.5 pro, and later summed up, verified, and properly written by the Danus system, a specialized agent built on Rethlas and substantially more capable for fundamental mathematical research based on the Rethlas system. Human verification and polishing were done afterwards. See \cite{Ju26} for a detailed introduction to the Rethlas system. Due to the limitation of automated systems, it is possible that we have missed some related references in the literature, and we welcome any comments from experts.
\end{rem}

\subsection*{Acknowledgements}
The author was partially supported by the National Key R\&D Program of China \#\allowbreak 2024YFA1014400.
The author would like to thank the Rethlas team, namely Haocheng Ju, Jiedong Jiang, Shurui Liu, Guoxiong Gao, Yuefeng Wang, Zeming Sun, Bin Wu, Liang Xiao, and Bin Dong, for their contributions to the development of Rethlas and its customized version used for the problem studied in this paper. The author would like to thank Fulin Xu for verifying an earlier draft of this paper and providing many useful comments and suggestions. The author would like to thank Ruochuan Liu and Gang Tian for constant support and encouragement.
The author would also like to thank Supravat Sarkar for a friendly discussion and for pointing out the further example recorded in Remark~\ref{rem:sarkar}.

\section{Slope instability}\label{sec:slope}

In this section, we prove the slope instability of the bundle in Theorem~\ref{thm:main}.

Let $X=\PP^2_{\mathbb C}$, let $\Oo_X(1)$ be the hyperplane line bundle, and let $H=c_1(\Oo_X(1))$. We write $T_X$ for the tangent bundle of $X$ and set
\[
\mathcal F=T_X(-1)=T_X\otimes\Oo_X(-1),\qquad \Ee=\mathcal F\oplus\Oo_X.
\]
On $\PP^2$ one has $H^2=1$.

\begin{lem}\label{lem:chern}
The bundle $\mathcal F$ has rank $2$ and first Chern class $c_1(\mathcal F)=H$. Consequently $\Ee$ has rank $3$ and first Chern class $c_1(\Ee)=H$.
\end{lem}

\begin{proof}
The Euler sequence on $\PP^2$, twisted by $\Oo_X(-1)$, is the short exact sequence
\begin{equation}\label{eq:euler}
0\longrightarrow \Oo_X(-1)\longrightarrow \Oo_X^{\oplus 3}\longrightarrow \mathcal F\longrightarrow 0.
\end{equation}
Hence $\mathcal F$ has rank $3-1=2$, and its total Chern class is
\[
c(\mathcal F)=\frac{c\bigl(\Oo_X^{\oplus 3}\bigr)}{c\bigl(\Oo_X(-1)\bigr)}=\frac{1}{1-H}=1+H+H^2
\]
in the Chow ring of $\PP^2$. In particular $c_1(\mathcal F)=H$. Since $\Ee=\mathcal F\oplus\Oo_X$ and $\Oo_X$ has rank $1$ and trivial first Chern class, $\Ee$ has rank $3$ and $c_1(\Ee)=c_1(\mathcal F)=H$.
\end{proof}

\begin{prop}\label{prop:unstable}
The bundle $\Ee$ is slope-unstable with respect to $H$.
\end{prop}

\begin{proof}
By Lemma~\ref{lem:chern} and the equality $H^2=1$, the $H$-slope of $\mathcal F$ is
\[
\mu_H(\mathcal F)=\frac{c_1(\mathcal F)\cdot H}{\rank(\mathcal F)}=\frac{H^2}{2}=\frac{1}{2},
\]
while the $H$-slope of $\Ee$ is
\[
\mu_H(\Ee)=\frac{c_1(\Ee)\cdot H}{\rank(\Ee)}=\frac{H^2}{3}=\frac{1}{3}.
\]
The inclusion of the first direct summand exhibits $\mathcal F$ as a nonzero proper subbundle of $\Ee$, and $\mu_H(\mathcal F)=\tfrac12>\tfrac13=\mu_H(\Ee)$. Hence $\Ee$ is slope-unstable with respect to $H$.
\end{proof}

\section{The two boundary classes}\label{sec:classes}

In this section we set up the geometry of $Y=\PP_X(\Ee)$ and identify two distinguished divisor classes on it.

Keep the notation of Section~\ref{sec:slope}, and set
\[
Y=\PP_X(\Ee),\qquad \pi\colon Y\to X
\]
for the projective bundle and its structure morphism. Let $\Oo_Y(1)$ be the tautological quotient line bundle on $Y$, and put
\[
A=\pi^*H\in N^1(Y)_{\RR},\qquad B=c_1(\Oo_Y(1))\in N^1(Y)_{\RR}.
\]
Since $\rank(\Ee)=3$ and $\dim X=2$, the variety $Y$ has dimension
\begin{equation}\label{eq:dimY}
\dim Y=\dim X+\rank(\Ee)-1=2+3-1=4.
\end{equation}

\begin{lem}\label{lem:rho}
The real N\'eron--Severi space $N^1(Y)_{\RR}$ is two-dimensional and is spanned by $A$ and $B$, which are linearly independent.
\end{lem}

\begin{proof}
The Picard group of a projective bundle is given by
\[
\Pic(Y)=\pi^*\Pic(X)\oplus \ZZ\,[\Oo_Y(1)].
\]
Since $\Pic(\PP^2)$ has rank $1$, the group $\Pic(Y)$ has rank $2$, and therefore $N^1(Y)_{\RR}$ is 
spanned by $A$ and $B$. 

To see that $A$ and $B$ are linearly independent in $N^1(Y)_{\RR}$, let $L\cong\PP^1$ be a line contained in a fiber $\pi^{-1}(x)$. Then $A\cdot L=(\pi^*H)\cdot L=0$ because $\pi$ contracts $L$ to the point $x$, while $B\cdot L=c_1(\Oo_Y(1))\cdot L=1$ since $\Oo_Y(1)$ restricts to the hyperplane bundle on the fiber. Hence $B$ is not a real multiple of $A$. Also $A$ is not numerically equivalent to $0$ because it is a pullback of an ample divisor via a surjective morphism.

 Since $N^1(Y)_{\RR}$ is 
spanned by $A$ and $B$, which are linearly independent, $N^1(Y)_{\RR}$ is two-dimensional. 
\end{proof}

\begin{lem}\label{lem:A}
The class $A$ is nef and not big.
\end{lem}

\begin{proof}
The class $A=\pi^*H$ is the pullback of the ample, hence nef, class $H$ on $X$, and the pullback of a nef class under a morphism is nef. Thus $A$ is nef.

By \eqref{eq:dimY} we have $\dim Y=4$, and
\[
A^4=(\pi^*H)^4=\pi^*(H^4)=0,
\]
because $H^4=0$ already on the surface $X$ (indeed $H^k=0$ for $k>2$). For a nef divisor class on a projective variety, bigness is equivalent to having positive top self-intersection number. Since $A^4=0$, the class $A$ is not big.
\end{proof}

\begin{lem}\label{lem:B}
The class $B$ is nef and not big.
\end{lem}

\begin{proof}
By the twisted Euler sequence \eqref{eq:euler}, the bundle $\mathcal F$ is generated by the three global sections coming from $\Oo_X^{\oplus 3}$. The structure sheaf $\Oo_X$ is generated by one global section. Therefore $\Ee=\mathcal F\oplus\Oo_X$ is generated by four global sections; equivalently, there is a surjection
\[
\Oo_X^{\oplus 4}\longrightarrow \Ee.
\]
In the Grothendieck convention, this surjection induces a morphism
\[
\varphi\colon Y=\PP_X(\Ee)\longrightarrow \PP^3
\]
with $\Oo_Y(1)=\varphi^*\Oo_{\PP^3}(1)$. Hence $B=\varphi^*c_1(\Oo_{\PP^3}(1))$ is the pullback of the ample, hence nef, hyperplane class of $\PP^3$, and is therefore nef.

The image of $\varphi$ is a subvariety of $\PP^3$, so it has dimension at most $3$, while $\dim Y=4$ by \eqref{eq:dimY}. Consequently
\[
B^4=\varphi^*\bigl(c_1(\Oo_{\PP^3}(1))^4\bigr)=0,
\]
since $c_1(\Oo_{\PP^3}(1))^4=0$ on the threefold $\PP^3$. As $B$ is nef and $B^4=0$, the class $B$ is not big.
\end{proof}

\section{Proof of Theorem~\ref{thm:main}}\label{sec:proof}

We now combine the results of the previous sections.

\begin{prop}\label{prop:1homog}
With the notation of Section~\ref{sec:classes}, the variety $Y=\PP_X(\Ee)$ is $1$-homogeneous, and
\begin{equation}\label{eq:cones}
\Nef(Y)=\Eff(Y)=\RR_{\geq 0}\,A+\RR_{\geq 0}\,B.
\end{equation}
\end{prop}

\begin{proof}
By Lemma~\ref{lem:rho} the space $N^1(Y)_{\RR}$ is two-dimensional, spanned by the linearly independent classes $A$ and $B$. By Lemmas~\ref{lem:A} and~\ref{lem:B}, both $A$ and $B$ are nef and not big.

Every nef class is pseudo-effective: a nef class is a limit of ample classes, and ample classes are big, hence pseudo-effective; passing to the limit keeps the class in the closed cone $\Eff(Y)$. Thus $A$ and $B$ are nonzero classes lying in $\Eff(Y)$. The big cone is the interior of $\Eff(Y)$, so, since $A$ and $B$ are not big, they lie on the boundary of $\Eff(Y)$.

The cone $\Eff(Y)$ is a closed convex salient cone with nonempty interior in the two-dimensional vector space $N^1(Y)_{\RR}$. In a two-dimensional space, such a cone has exactly two boundary rays, and these rays are determined by any two linearly independent boundary classes. As $A$ and $B$ are linearly independent boundary classes, the two boundary rays of $\Eff(Y)$ are $\RR_{\geq 0}\,A$ and $\RR_{\geq 0}\,B$, and therefore
\[
\Eff(Y)=\RR_{\geq 0}\,A+\RR_{\geq 0}\,B.
\]

Since $A$ and $B$ are nef and $\Nef(Y)$ is a convex cone, the span $\RR_{\geq 0}\,A+\RR_{\geq 0}\,B$ is contained in $\Nef(Y)$. Conversely, $\Nef(Y)\subseteq\Eff(Y)$ for any projective variety. Combining these inclusions with the previous paragraph gives
\[
\RR_{\geq 0}\,A+\RR_{\geq 0}\,B\subseteq\Nef(Y)\subseteq\Eff(Y)=\RR_{\geq 0}\,A+\RR_{\geq 0}\,B,
\]
so all three cones coincide. This proves \eqref{eq:cones}, and in particular $\Nef(Y)=\Eff(Y)$, so $Y$ is $1$-homogeneous.
\end{proof}

\begin{proof}[Proof of Theorem~\ref{thm:main}]
By Lemma~\ref{lem:chern}, the bundle $\Ee=T_{\PP^2}(-1)\oplus\Oo_{\PP^2}$ has rank $3$. By Proposition~\ref{prop:unstable}, it is slope-unstable with respect to $H$. By Proposition~\ref{prop:1homog}, its projectivization $\PP(\Ee)$ is $1$-homogeneous. Taking $X=\PP^2_{\mathbb C}$ and this $\Ee$ therefore answers Question~\ref{ques:fl} affirmatively.
\end{proof}

\end{document}